\documentclass[11pt]{article}

\usepackage{graphicx}
\usepackage{psfrag}
\usepackage{epsf}
\usepackage{subfigure}
\usepackage{amsmath,amsfonts,amssymb,latexsym}
\usepackage{enumitem}
\usepackage{url}
\usepackage[cp1250]{inputenc}
\usepackage{setspace}

\newcommand{\ignore}[1]{}
\newcommand{\startClaims}{\setcounter{claim}{0}}
\newtheorem{theorem}{Theorem}
\newtheorem{corollary}[theorem]{Corollary}
\newtheorem{lemma}[theorem]{Lemma}

\newenvironment{proof}%
{\noindent{\bf Proof.}\ }%
{\hfill\eopf\par\bigskip}%

{\noindent{\bf Sketch of a proof.}\ }%
{\hfill\eopf\par\bigskip}%
{\noindent{\bf Ideas for a proof.}\ }%
{\hfill\eopf\par\bigskip}%
{\noindent{\bf Proof in progress.}\ }%
{\hfill\eopf\par\bigskip}%

\newenvironment{proofof}[1]
{\medskip\noindent{\bf Proof of #1.}\ }
{\hfill\eopf\par\bigskip}

\newcommand{\dfnc}[3]{#1:#2\rightarrow #3}
\newcommand{\dset}[2]{\left\{#1 \:|\: #2\right\}}

\def\i4c{{internally-4-connected}}
\def\2cc{{2-crossing-critical}}

\def\m2{{{\cal M}_2}}

\def\myco#1{}

\newcommand{\crn}{\operatorname{cr}}
\newcommand{\mcrn}{\operatorname{mcr}}
\newcommand{\crs}{\operatorname{cr}_\Sigma}
\newcommand{\mcrs}{\operatorname{mcr}_\Sigma}

\newcommand{\eopf}{\raisebox{0.8ex}{\framebox{}}}

\title{Minor crossing number is additive over arbitrary cuts}

\author{
Drago Bokal, \\
Faculty of Natural Sciences and Mathematics,\\
University of Maribor, Slovenia,\\
\small{\texttt{drago.bokal@uni-mb.si}}
\and
Markus Chimani, \\
Faculty of Mathematics and Computer Science,\\
Friedrich-Schiller-University Jena, Germany,\\
\small{\texttt{markus.chimani@uni-jena.de}}
\and
Jes\'us Lea\~{n}os,  \\
Academic Unit of Mathematics,\\
Autonomous University of Zacatecas, M\'exico,\\
\small{\texttt{jelema@uaz.edu.mx}}
}

\begin{document}
\maketitle


\begin{abstract}
We prove that if $G$ is a graph with an minimal edge cut $F$ of size three and
$G_1$, $G_2$ are the two (augmented) components of $G-F$, then the crossing
number of $G$ is equal to the sum of crossing numbers of $G_1$ and $G_2$.
Combining with known results, this implies that crossing number is additive over
edge-cuts of size $d$ for $d\in\{0, 1, 2, 3\}$, whereas there are
counterexamples for every $d\ge 4$. The techniques generalize to show that minor
crossing number is additive over edge cuts of arbitrary size, as well as to
provide bounds for crossing number additivity in arbitrary surfaces. We point
out several applications to exact crossing number computation and crossing critical
graphs, as well as provide a very general lower bound for the minor crossing 
number of the Cartesian product of an arbitrary graph with a tree.
\end{abstract}

\section{Introduction}

We consider the problem of finding, or at least bounding, the crossing number of
a graph $G$ based on the crossing numbers of its components when decomposing $G$
via small edge cuts. We assume that the reader is familiar with the concept of
crossing numbers of graphs in surfaces: each crossing of non-adjacent edges in a
drawing counts. Let $G$ be a graph and $\Sigma$ a surface, then $\crs(G)$
denotes the minimum number of crossings of some drawing of $G$ in $\Sigma$. 
We further consider a related concept, minor crossing numbers, to which
our techniques also apply. For a graph $G$ and a surface $\Sigma$, the minor
crossing number of $G$ in $\Sigma$ is the minimum crossing number in $\Sigma$
over all graphs that have $G$ as a minor: $\mcrs(G)=\min_{G\preceq H}\crs(H).$
A graph $H$ yielding equality in this definition is said to be a realizing
graph of $G$, its optimal drawing is a realizing drawing. Intuitively, this concept
allows for further minimization of the number of crossings in a drawing of $G$ by
replacing each vertex of $G$ with a tree. More can be found in \cite{BFM}, where 
the concept was introduced, or in \cite{BCSV}, where an embedding method,
sharing some intuitive background with our methods, is presented in the context
of the minor crossing number. For both crossing number concepts, we may omit 
the subscript when considering the sphere or, equivalently, the plane.

Let $G=(V,E)$ be a connected graph and $F\subseteq E(G)$ a cut
in $G$ of size $d=|F|$. Let $H_1$ and $H_2$
be the two components of $G-F$. When studying a graph invariant, it is natural
to ask, how does the value of that invariant on $G$ depend on the values on
$H_1$ and $H_2$. When considering this question for crossing numbers, we need to
define auxiliary graphs $G_i=G/H_{3-i}$, obtained from $G$ by contracting
$H_{3-i}$, for $i=1,2$. Note that $G_i$ is also obtained from $H_i$ by adding a
new vertex and connecting it to all the endvertices of $F$ in $H_i$.

We can view such cuts also in an inverse way, leading to the technically
stronger concept of \emph{zip products} of graphs. Introduced in~\cite{DB}, we
consider it here in a version generalized from simple to arbitrary
(multi)graphs. For $i=1,2$, let $G_i$ be a graph with a vertex $v_i$ of degree
$d$, whose adjacent edges in $G_i$ form the set $F_i$. Let
$\dfnc{\sigma}{F_1}{F_2}$ be any bijection, and let $G$ be the graph obtained
from the disjoint union of $G_1-v_1$ and $G_2-v_2$ by adding the edges $vw$ for
each $vv_1\in F_1$ and corresponding $wv_2=\sigma(vv_1)\in F_2$. We may denote
these new edges as $F$. We say $G$ is the zip product of $G_1$ and $G_2$
at $v_1$ and $v_2$, respectively, for bijection $\sigma$. For the rest of
the paper, we refer to the edges and vertices of $G$ belonging to the subgraph
$G_1-v_1$ ($G_2-v_2$) as \emph{green} (\emph{red}, respectively), and to the
edges $F$ as \emph{blue}. 

A bundle $B$ of a vertex $v$ in $G$ is a union of $d_G(v)$ pairwise edge
disjoint paths in $G-v$, where $d_G(v)$ denotes the degree of $v$ in $G$: all
these paths start in the neighborhood $N_G(v)$ of $v$ and end at some (common)
fixed vertex in $G$, denoted as the \emph{sink} of $B$. In particular, the
number of paths starting at any $u\in N_{G}(v)$ equals the number of edges
between $u$ and $v$ in $G$. Two bundles of $v$ in $G$ are \emph{coherent}, if
they have distinct sinks and are edge disjoint. We can observe that the edges
$F$ arising from the zip product of $G_1$ and $G_2$ are a minimum cut separating
$w_1$ from $w_2$ whenever, for $i=1,2$, $G_i$ has a bundle of $v_i$ with sink
$w_i$. Inversely, any minimum cut $F$ gives rise to a corresponding zip product,
and, whenever $|F|\le 3$, there always exists at least one corresponding bundle 
in each component. Our main result is the following:

\begin{theorem}
\label{th:mainCrn}
Let $\Sigma$ be an arbitrary surface and let $G$ be a zip product of $G_1$ and
$G_2$ at $v_1$ and $v_2$, respectively. If $d_{v_i}(G_i)\le 3$, or if each of
$v_1$ and $v_2$ has two coherent bundles in $G_1$ and $G_2$, respectively, then
$\crs(G)\ge \crs(G_1)+\crs(G_2)$.
\end{theorem}

Theorem~\ref{th:mainCrn} generalizes the following result,
as well as removes the two-bundle condition for small cuts:

\begin{theorem}[\cite{DB}]
\label{th:twoBundles}
Let $G$ be a zip product of $G_1$ and $G_2$ at $v_1$ and $v_2$ with
$d(v_1)=d(v_2)$. If each of $v_1$ and $v_2$ has two coherent bundles in $G_1$
and $G_2$, respectively, then $\crn(G)\ge \crn(G_1)+\crn(G_2)$.
\end{theorem}

As the counterexamples to the claim of Theorem~\ref{th:mainCrn} in the presence
of just one bundle at each $v_i$ are exhibited in \cite{BB} for any
$d_{G_i}(v_i)\ge 4$, our result closes the question of additivity of crossing
numbers over cuts with at most one bundle at each vertex. Furthermore, our
approach gives an alternative proof of Theorem~\ref{th:twoBundles} that allows
for generalization into higher surfaces. Our methods also generalize to the
minor crossing number, establishing the following:

\begin{theorem}
\label{th:mainMcrn}
Let $\Sigma$ be an arbitrary surface and let $G$ be a zip product of $G_1$ and
$G_2$ at $v_1$ and $v_2$, respectively. Then,
$\mcrs(G)\ge\mcrs(G_1)+\mcrs(G_2)$.
\end{theorem}

Note that for the minor crossing number, no bundles are required for the
additivity of lower bounds. Furthermore, additivity of minor crossing number
over blocks of a graph is established in \cite{BFM}, and
Theorem~\ref{th:mainMcrn} is a generalization of that result. Also, relationships
of minor crossing number and bisection width have been studied in \cite{BCSV}; the
major difference here is that the crossing number is estimated in terms of the 
minor crossing number of the two graphs resulting from the cut, but in the 
bisection width method, the lower bound is given in terms of the size of the smallest cut
splitting the graphs into roughly equal parts. However, the embedding method from
the same paper does give the bound in terms of the crossing number of the embedded
graph, and our result could be considered a refinement of that method. 
In the proof of Theorem \ref{th:mainMcrn}, we essentialy find a specific embedding
of the disjoint union of $G_1$ and $G_2$ into $G$, yielding the desired lower bound.

\section{Auxiliary lemmata}

We first state some key ingredients needed in our proofs of
Theorems~\ref{th:mainCrn} and~\ref{th:mainMcrn}. If $\times=(e,f)$ is a crossing
of $e,f\in E(G)$ in some drawing of $G$, then we denote by $G^\times$ the graph
obtained by subdividing $e$ and $f$ and identifying the two new vertices. 

\begin{lemma}
\label{lm:observation}
Let $\Sigma$ be an arbitrary surface and let $G^{(e,f)}$ be obtained from
$G=(V,E)$ by subdividing two distinct edges $e,f\in E(G)$ and identifying the
new vertices into a vertex $x$. Then $\crs(G^{(e,f)})\ge \crs(G)-1$. Moreover,
if $e$ and $f$ cross in some optimal drawing $D$ of $G$ in $\Sigma$, then we
have equality.
\end{lemma}
\begin{proof}
Suppose not. So there would be a drawing of $G^{(e,f)}$ with at most $\crs(G)-2$
crossings. Then we could reintroduce the crossing instead of the vertex $x$ to
obtain a drawing of $G$ with at most $\crs(G)-1$ crossings, a contradiction.

Now if $D$ is an optimal drawing of $G$ in $\Sigma$, we can place $x$ at the
same point as the crossing between $e$ and $f$ and obtain a drawing of
$G^{(e,f)}$ with $\crn(G)-1$ crossings, yielding the lower bound. 
\end{proof}

Recall that $\Sigma=\Sigma_1\#\Sigma_2$ denotes the connected sum of two
surfaces, and that if $\Sigma$ is a sphere, then so are $\Sigma_1$ and
$\Sigma_2$. The following lemma will help us establishing a (hypothetical)
minimum counterexample to our main theorems.

\begin{lemma}
\label{lm:redGreen}
Let $\Sigma$ be a surface, let $G$ be a zip product of $G_1$ and
$G_2$ at vertices of degree $d$, such that (i)
$\crs(G)<\min_{\Sigma=\Sigma_1\#\Sigma_2}\crn_{\Sigma_1}(G_1)+\crn_{\Sigma_2}(G_2)$
and (ii) $G$ has the smallest crossing number among the graphs with these
properties. If $D$ is an optimal drawing of $G$ in $\Sigma$, then any crossing
in $D$ is a red-green crossing (i.e., a crossing between a red and a green
edge).
\end{lemma}
\begin{proof}
Assume that $D$ has a crossing $\times$ not of the type red-green. In each of
the following cases, we find an alternative to graph $G$ with smaller crossing
number, a contradiction required to establish the claim. 

First, assume that $\times$ is a green-green crossing. By Lemma
\ref{lm:observation}, $\crs(G^\times)=\crs(G)-1$ and
$\crn_{\Sigma_1}(G_1^\times)\ge \crn_{\Sigma_1}(G_1)-1$. As $G^\times$ is a zip
product of $G_1^\times$ and $G_2$, we have
$\crs(G^\times)=\crs(G)-1\le\crn_{\Sigma_1}(G_1)-1+\crn_{\Sigma_2}(G_2) \le
\crn_{\Sigma_1}(G_1^\times)+\crn_{\Sigma_2}(G_2)$. Since the argument applies to
arbitrary $\Sigma=\Sigma_1\#\Sigma_2$, $G^\times$ contradicts the choice of $G$.
Similarly, we can show that $D$ has no crossings of type red-red. 

Second, assume that $\times$ is a crossing between a green edge $e$ and a blue
edge with green endvertex $v$. Now, $G^\times$ is a zip product of
$G_1^{(e,vv_1)}$ and $G_2$, and a similar contradiction as before applies.
Similarly, we can show that $D$ has no crossings of type blue-red. 

Third, let $\times$ be a crossing of two blue edges with green endvertices $v$
and $w$ (note that, by the optimality of $D$, $v\not=w$). The graph
$G_1^{(vv_1,wv_1)}$ has a double edge $xv_1$, where $x$ is the new vertex. The
graph $G^\times$ is a zip product of $G_1^{(vv_1,wv_1)}$ and $G_2$, which has
crossing number equal to $\crn(G)-1$ by Lemma \ref{lm:observation}, a final
contradiction to the choice of $G$.
\end{proof}

\begin{lemma}
\label{lm:stars}
Let $H_1=(V,E)$ be a graph embedded in some surface $\Sigma$, let $H_2$ be its
dual in $\Sigma$, and, for $i=1,2$, let $T_i \le H_i$ be an arbitrary tree.
Then, $|E(T_1)|+|E(T_2)|\le |E(H_1)|$.
\end{lemma}
\begin{proof}
Let $f$,$m$, and $n$ be the number of faces, edges, and vertices, respectively,
of $H_1$. Due to duality, $H_2$ has $f$ vertices, $m$ edges, and $n$ faces. By
Euler's formula, $m=n+f-2+g$, where $g$ is the genus of $\Sigma$. As $T_1$ and
$T_2$ live in different graphs, they are totally disjoint. If $k$ is their total
number of vertices, then $k\le n+f$ and they have $k-2\le n+f-2 \le m =
|E(H_1)|$ edges.  
\end{proof}

\begin{lemma}
\label{lm:upper}
Let $\Sigma$ be a surface, assume that $\Sigma=\Sigma_1\#\Sigma_2$, and let $G$
be a zip product of $G_1$ and $G_2$ at vertices of degree at most three w.r.t.\ 
some bijection $\sigma$. Then, (i)
$\crs(G)\le\crn_{\Sigma_1}(G_1)+\crn_{\Sigma_2}(G_2)$ and (ii)
$\mcrs(G)\le\mcrn_{\Sigma_1}(G_1)+\mcrn_{\Sigma_2}(G_2)$.
\end{lemma}
\begin{proof}
First we prove (i). For $i=1,2$, let $D_i$ be an optimal drawing of $G_i$ in
$\Sigma_i$. Let $N_i$ be a small disk around $v_i$, such that $D_i\cap N_i$ is a
star. We can obtain a drawing of $G$ in $\Sigma$ with
$\crn_{\Sigma_1}(G_1)+\crn_{\Sigma_2}(G_2)$ crossings by identifying the
surfaces $\Sigma_i\setminus N_i$ along the boundaries of $\partial N_i$ such
that the edges originally adjacent to $v_1$ or $v_2$ match up according to
$\sigma$. Note that we may need to mirror $D_2$ to match the vertex rotation of
$v_2$ with the one of~$v_1$.

For (ii), observe that any realizing graph of $G_i$ has a cubic vertex $v_i'$ in
the tree representing $v_i$. A drawing of a graph with $G$ minor that
establishes the claimed upper bound can thus be obtained from arbitrary
realizing drawings $D_i$ of $G_i$ in $\Sigma_i$ following the same steps as in
the proof of~(i).
\end{proof}

\section{Additivity theorems and consequences}

In this section we prove Theorems \ref{th:mainCrn} and \ref{th:mainMcrn}.
Although the proof of the former could follow the same steps as the proof of the
latter, we provide independent proofs for clarity.

\begin{proofof}{Theorem \ref{th:mainCrn}}
For $d=0,1$, the statement is trivial. Although the following arguments also
apply for $d=2$, this case has been known before \cite{LS}.
Therefore, we may assume $d=3$, or $d\ge 4$ and each of $v_1$ and $v_2$ has two
coherent bundles. Let $G$ be a counterexample with smallest crossing number and
let $D$ be an optimal drawing of $G$. By Lemma \ref{lm:redGreen}, each crossing
in $D$ is a crossing of a green and a red edge. Let $D'$ be the drawing obtained
from $D$ by (i) adding some uncrossed dotted green (red) edges in the interior of the faces 
of the red (green) drawing, so that the green (red) graph, induced by a red (green)
face is connected, (ii) contracting all green and red edges that do not cross 
(note that all dotted edges are now contracted) and (iii)
subdividing every edge of $D$ that is crossed several times. Hence every edge in
$D'$ is crossed precisely once, and every crossing is still of type green-red.
Then $D'$ induces two graphs, $H_1$ and $H_2$, spanned by green and red
edges, respectively, and embedded in $\Sigma$. They are duals of each other, 
hence have the same number of edges, and the number of crossings of $D$ 
and $D'$ is equal to this number of edges. Furthermore, the possible two coherent 
bundles in $G_i$ contract to coherent bundles in $H_i$. Let $S_1$ and $S_2$ 
be the set of green and red endvertices of blue edges $F$, respectively.

For $d\le 3$ and $i=1,2$, let $T_i$ be a tree in $H_i$ containing all the
vertices of $S_i$. For $i=1,2$, let $D_i$ be the subdrawing of $D$ spanned by
$(G_i-v_i)\cup F$ and merged with the subdrawing of $D'$ spanned by $T_{3-i}$.
The total number of crossings in $D_1$ and $D_2$ equals the number of edges in
$T_1\cup T_2$. By Lemma \ref{lm:stars}, this is at most $|E(H_1)|=|E(H_2)|$,
which is equal to $\crs(D')=\crs(D)=\crs(G)$. 

In $D_i$, for $i=1,2$, we can contract the nodes $S_{3-i}$ along $T_{3-i}$ into
a single vertex $v_i$ to obtain a drawing $D_i'$ of $G_i$ with
$\crs(D_i)=\crs(D_i')$. As $\crs(G)=\crs(D)\ge \crs
(D_1')+\crs(D_2')\ge\crs(G_1)+\crs(G_2)$, $G$ is not a counterexample, a
contradiction establishing the claim.

For $d\ge 4$ and $i=1,2$, the graph $H_i$ has two coherent (i.e., edge disjoint)
bundles starting at the vertices of $S_i$. Let $B_i$ be one with less than half
of the edges of $H_i$. Let $D_i$ be the subdrawing of $D$ spanned by
$(G_i-v_i)\cup F$ and merged with the subdrawing of $D'$ spanned by $B_{3-i}$.
We may assume that $D_i[B_{3-i}]$ is a drawing of a $d$-star, as otherwise we
can split each vertex of $B_{3-i}$ that is common to two of the bundle paths in
its small neighborhood and route the edges of the paths properly to satisfy the
assumption. The total number of crossings in $D_1$ and $D_2$ equals the number
of edges in $B_1\cup B_2$. As each bundle has at most half of the edges of $H_i$
and $|E(H_1)|=|E(H_2)|$, the drawings $D_1$ and $D_2$ have together at most
$|E(H_i)|=\crs(D')=\crs(D)=\crs(G)$ crossings. 

By contracting the edges of $B_{3-i}$ into a single vertex $v_i$ in $D_i$, for
$i=1,2$, we obtain drawings $D_i'$ of $G_i$ with $\crs(D_i)=\crs(D_i')$. As
$\crs(G)=\crs(D)\ge \crs (D_1')+\crs(D_2')\ge\crs(G_1)+\crs(G_2)$, $G$ is not a
counterexample, a contradiction establishing the claim.
\end{proofof}

Following essentially similar ideas as in the proof of Theorem~\ref{th:mainCrn},
we can prove Theorem~\ref{th:mainMcrn}. The major difference is that we substitute 
the minimum counterexample argument by a more technical treatment of the crossings 
involving blue edges. This could also be done in the previous proof, but at the expense 
of its clarity.

\begin{proofof}{Theorem \ref{th:mainMcrn}}
For $d=0,1$, the statement is established in \cite{BFM}. By induction on $d$, we
may assume that $G_i-v_i$ is connected. 
Let $G$ be as in the statement. Let $G'$ be its realizing graph, and let $D$ be
an optimal drawing of $G'$ in $\Sigma$, i.e. an realizing drawing of $G$. Let $F$
be the set of blue edges, and let $b_1$ (respectively, $b_2$) be the number of
crossings in $G'$ that involve a blue and a green (respectively, red) edge.
There is a natural extension of the red and green colors from $G$ to $D$: any
vertex or edge of $G'$ corresponding to a vertex or edge of $G_1$ is green,
those corresponding to $G_2$ are red, and any crossing of two green (red) edges
is green (red, respectively).

To obtain $D'$ from $D$, we first remove the blue edges, introduce vertices at
all monochromatic crossings, contract all green and red edges that in the
subsequent drawing are not crossed, and properly subdivide every edge of $D$
that is crossed several times. Hence every edge in $D'$ is crossed precisely
once, and every crossing is of a green and a red edge. Then, $D'$ induces two
graphs embedded in $\Sigma$, $H_1$ and $H_2$, spanned by green and red edges,
respectively. These graphs are duals of each other, have the same number of
edges, and the number of crossings of $D'$ is equal to this number of
edges.

Let $S_1$ and $S_2$ be the set of green and red endvertices of blue edges,
respectively, and let $T_i$ be a tree in $H_i$ containing all the vertices of
$S_i$. For $i=1,2$, let $D_i'$ be the subdrawing of $D'$ spanned by $H_i\cup
T_{3-i}$ and augmented as follows: (a) we add the $F$-segments from the drawing
$D$, (b) we split any crossing of two $F$-edges by rerouting the crossing paths
(preserving the fact that $F\cup T_i$ is connected), and (c) in $D_1'$
(respectively, $D_2'$), we only maintain the segment of the blue edge connecting
its green (respectively, red) endvertex with the first red (green) point in the
drawing (which is either a crossing with a red (green) edge or the red (green)
endvertex; thus all green (red) endvertices of blue edges are connected to the
tree, but the blue edges never cross the red). The total number of crossings in
$D_1'$ and $D_2'$ equals the number of edges in $T_1\cup T_2$, increased by
$b_1+b_2$. By Lemma \ref{lm:stars}, this is at most
$|E(H_1)|+b_1+b_2\le\crn(D')+b_1+b_2$. Let $D_i$ be obtained by uncontracting
the previously contracted subdrawings of $D$ within a small neighborhood of
their corresponding vertices in $D_i'$. Any crossing in $D_i$ but not in $D_i'$
exists in $D$ but not in $D'$ and let $r$ be the number of such crossings. Then,
$\crn(D_1)+\crn(D_2)=\crn(D_1')+\crn(D_2')+r\le\crn(D')+b_1+b_2+r\le\crn(D)$.

Then, $D_i'$ is a (not necessarily optimal) drawing of a graph that has $G_i$
as a minor. Therefore, $\mcrs(G_i)\le\crn(D_i')$ and the claim follows.
\end{proofof}

We state some corollaries that easily follow from the above theorems.

\begin{corollary}
\label{cr:sandwich}
Let $G$ be a graph, and let $F\subseteq E(G)$ be a minimal edge cut of $G$. Let
$G_i$, $i=1,2$, be obtained from the two components $H_i$ of $G-F$ by adding to
each of them a new vertex $v_i$ and connecting it to the endvertices of $F$ in
$H_i$. If $|F|\le 3$, then 
\begin{align*}
\crs(G_1)+\crs(G_2)\le\crs(G)&\le\min_{\Sigma=\Sigma_1\#\Sigma_2}(\crn_{\Sigma_1}(G_1)+\crn_{\Sigma_2}(G_2)),\\
\mcrs(G_1)+\mcrs(G_2)\le\mcrs(G)&\le\min_{\Sigma=\Sigma_1\#\Sigma_2}(\mcrn_{\Sigma_1}(G_1)+\mcrn_{\Sigma_2}(G_2)).
\end{align*}
\end{corollary}

\begin{proof}
Combine the lower bounds of Theorems~\ref{th:mainCrn} and~\ref{th:mainMcrn} with
the upper bound in Lemma~\ref{lm:upper}.
\end{proof}

\begin{corollary}
\label{cr:equalityForSphere}
Let $G$ be a graph, and let $F\subseteq E(G)$ be a minimal edge cut of $G$. Let
$G_i$, $i=1,2$, be obtained from the two components $H_i$ of $G-F$ by adding to
each of them a new vertex $v_i$ and connecting it to the endvertices of $F$ in
$H_i$. If $|F|\le 3$, then
\begin{align*}
\crn(G)&=\crn(G_1)+\crn(G_2),\\
\mcrn(G)&=\mcrn(G_1)+\mcrn(G_2).
\end{align*}
\end{corollary}
\begin{proof}
Combine Corollary \ref{cr:sandwich} with the observation, that whenever $\Sigma$
is the sphere and $\Sigma=\Sigma_1\#\Sigma_2$, then both $\Sigma_1$ and
$\Sigma_2$ are spheres.
\end{proof}

The reader will easily see that Corollary \ref{cr:equalityForSphere} implies the
desired crossing number of $G$ in Corollaries \ref{cr:criticalZip} and
\ref{cr:criticalCover}. Arguments from \cite{DB2}, which we do not repeat here,
establish the criticality of $G$: the crucial fact is that zipping of a critical
graph makes the edges involved in the zip product crossing critical.

\begin{corollary}
\label{cr:criticalZip}
For $i=1,2$, let $G_i$ be a $k_i$-crossing critical graph and $v_i\in V(G_i)$
such that $d_{G_1}(v_1)=d_{G_2}(v_2)\le 3$. If $G$ is any zip product of $G_1$ and $G_2$ 
at $v_1$ and $v_2$, then $G$ is a $k_1+k_2$-critical graph.
\end{corollary}

\begin{corollary}
\label{cr:criticalCover}
Let $G$ be any graph and let $S\subset V(G)$ be a vertex cover of $G$ containing
only vertices of degree $2$ and $3$. For each $v\in S$, let $G_v$ be a
$k_v$-critical graph with a vertex $u_v\in V(G_v)$ of degree $d_G(v)$. Let $G^S$
be the graph obtained from $G$ by iteratively zipping the graphs $G_v$ with $G$
at vertices $v$ and $u_v$. Then, $G^S$ is a $k$-critical graph for
$k=\crn(G)+\sum_{v\in S}k_v$.  
\end{corollary}

\begin{corollary}
Let $G$ be any graph with $m$ edges and crossing number $r$. For any $k\ge
m+r+1$, there exists an infinite family of $k$-crossing-critical graphs that all
contain $G$ as a subdivision. 
\end{corollary}
\begin{proof}
Let $G'$ be obtained from $G$ by subdividing every edge. The new vertices all
have degree two and form a vertex cover $S$ of $G'$. For a selected vertex $w\in
S$, let $G_w$ be any graph from the infinite family of $2$-crossing-critical
graphs, constructed by Kochol in \cite{K}, with one of its edges subdivided by a
vertex $u_w$. For any $v\in S\setminus \{w\}$, let $G_v$ be a $K_{3,3}$ with one
edge subdivided by a vertex $u_v$. By Corollary \ref{cr:criticalCover}, $G^S$ is
a $r+m+1$-crossing-critical graph. For $l=k-r-m-1>0$, zipping $l$ copies of
$K_{3,3}$ to $G^S$ establishes the claim. 
\end{proof}

\begin{corollary}
\label{coro:reverse}
Let $G$ be a 3-edge-connected crossing critical graph and let $F\subseteq E(G)$
be a minimal edge cut of $G$ of size $3$. Let $H_1$ and $H_2$ be the two
components of $G-F$ and, for $i=1,2$, let $G_i:=G/H_{3-i}$. Then there exists a
crossing critical graph $J_i$ with $H_i\subseteq J_i \subseteq G_i$.
\end{corollary}

\begin{proof}
It is easy to see that it is sufficient to show that, for $i=1,2$, the following
is true: if $e\in E(H_i)$, then $\crn(G_i-e)<\crn(G_i)$.

For a fixed $i\in\{1,2\}$, let $e$ any fixed edge of $H_i$. Since $G$ is
$3$-edge-connected, $H_i-e$ is connected. We can apply Theorem~\ref{th:mainCrn}
to each of the graphs $G$ and $G-e$ and obtain that
$\crn(G)=\crn(G_i)+\crn(G_{3-i})$ and $\crn(G-e)=\crn(G_i-e)+\crn(G_{3-i})$.
Now, note that these equations and the fact $\crn(G-e)<\crn(G)$, imply
$\crn(G_i-e)<\crn(G_i)$, as desired.
\end{proof}

The following result was established in \cite{LS}:

\begin{theorem} [\cite{LS}]
\label{coro:LS} 
Let $G$ be a connected crossing-critical graph with minimum degree at least $3$.
Then there is a collection $J_1, J_2 ,\ldots , J_\ell$ of $3$-edge-connected
crossing critical graphs, each of which is contained as a subdivision in $G$,
such that $\crn(G)= \sum_{i=1}^\ell \crn(J_i)$.
\end{theorem}

An appropriately repeated application of Corollaries~\ref{coro:LS}
and~\ref{coro:reverse} yields the following.

\begin{corollary}
\label{coro:decomposition} 
Let $G$ be a connected crossing-critical graph with minimum degree at least $3$.
Then there is a collection $I_1, I_2 ,\ldots , I_{\ell'}$ of
internally-$4$-edge-connected crossing critical graphs, each of which is
contained as a subdivision in $G$, such that $\crn(G)= \sum_{i=1}^{\ell'}
\crn(I_i)$. 
\end{corollary}


We conclude with a lower bound for the minor crossing number 
of the Cartesian product of an arbitrary graph with an arbitrary tree. 
Arguments of \cite{DB1} establish that $G\Box T$ can be obtained as a zip
product of graphs $\dset{G^{(d_G(v))}}{v\in V(T)}$, where $G^{(i)}$ 
denotes the join of $G$ with an independent set of $i$ vertices. Then, 
Theorem \ref{th:mainMcrn} establishes the following lower bound:

\begin{corollary}
Let $T$ be any tree and $G$ any graph. Then, $$\mcrs(T\Box G)\ge
\sum_{v\in V(T)}\mcrs(G^{(d_T(v))}).$$ 
\end{corollary}

\section*{Acknowledgement}
The authors like to express gratitude to the organizers of the BIRS workshop ``Crossing numbers turn useful'', as well as to the BIRS staff and the BANFF center, for the very stimulating workshop where several years of different approaches to the problem culminated in the above results. We also thank Bruce Richter for helpful discussions.



M.\ Chimani was funded via a Carl-Zeiss-foundation juniorprofessorship.
D.\ Bokal was funded through Slovenian Research Agency basic research projects J6-3600, J1-2043 and research programme P1-0297.

\end{document}